\documentclass[12pt, reqno]{amsart}
\usepackage{amsmath, amsthm, amscd, amsfonts, amssymb, graphicx, color}
\usepackage[bookmarksnumbered, colorlinks, plainpages]{hyperref}

\newtheorem{theorem}{Theorem}[section]

\newtheorem{proposition}[theorem]{Proposition}
\newtheorem{corollary}[theorem]{Corollary}
\theoremstyle{definition}

\newtheorem{example}[theorem]{Example}

\theoremstyle{remark}
\newtheorem{remark}[theorem]{Remark}
\numberwithin{equation}{section}

\newcommand{\HH}{\mathcal{H}}

\newcommand{\R}{\mathbb{R}}

\newcommand{\leqs}{\leqslant}
\newcommand{\geqs}{\geqslant}

\newcommand{\ap}{\alpha}

\newcommand{\ld }{\lambda}

\begin{document}
\setcounter{page}{1}

\title[Operator Monotone Functions on the Nonnegative Reals]{Integral Representations and Decompositions of Operator Monotone Functions on the Nonnegative Reals}

\author[P. Chansangiam]{Pattrawut Chansangiam$^1$$^{*}$}

\address{$^{1}$ Department of Mathematics, Faculty of Science, King Mongkut's Institute of Technology
Ladkrabang,  Bangkok 10520, Thailand.}
\email{\textcolor[rgb]{0.00,0.00,0.84}{kcpattra@kmitl.ac.th}}



\date{Received: xxxxxx; Revised: yyyyyy; Accepted: zzzzzz.
\newline \indent $^{*}$ Corresponding author}

\begin{abstract}

In this paper, we show that there is a one-to-one correspondence between operator monotone functions 
on the nonnegative reals and
finite Borel measures on the unit interval.
This correspondence appears as an integral representation of special operator monotone functions $x \mapsto 1\,!_t\,x$
for $t \in [0,1]$ with respect to a finite Borel measure on $[0,1]$, here $!_t$ denotes the $t$-weighted harmonic mean.
Hence such functions form building blocks for arbitrary operator monotone functions on the nonnegative reals. 
Moreover, we use this integral representation to decompose operator monotone functions.

\end{abstract}

\maketitle

\section{Introduction}

\noindent
A useful and important class of real-valued functions is the class of operator monotone functions,
introduced by L\"{o}wner in a seminal paper \cite{Lowner}.
These functions are functions of Hermitian matrices/operators preserving order.
More precisely, a real-valued function $f$, defined on an interval $I$, is said to be \emph{operator monotone}  
 if for all Hilbert spaces $\HH$ and
for all Hermitian operators $A,B$ on $\HH$ whose spectra
are contained in $I$, 
\begin{align*}
	A \leqs B \implies f(A) \leqs f(B),
\end{align*}
where $f(A)$ is the function calculus of $f$ at $A$. 
An operator monotone function from $\R^+=[0,\infty)$ to itself
is also known as a complete Bernstein function or a Nevanlinna-Pick function for the half-line.
In that paper, he characterized operator monotonicity in terms of the positivity of matrix of divided differences and an important class of analytic functions, namely, Pick functions.
This concept is closely related to operator convex/concave functions which was studied afterwards by Kraus in \cite{Kraus}.
A function $f: I \to \R$ is said to be \emph{operator concave} if for all Hilbert spaces $\HH$ and
for all Hermitian operators $A,B$ on $\HH$ whose spectra
are contained in $I$, 
\begin{align*}
	f(tA + (1-t)B) \geqs tf(A) + (1-t)f(B), \quad t \in [0,1].
\end{align*}  
Operator monotone functions and operator convex/concave functions
arise naturally in matrix and operator inequalities (e.g. \cite{Ando}, \cite{Bhatia},
\cite{Bhatia_positive def matrices}, \cite{Zhan}).
It is remarkable that a function defined on $\R^+$ is operator monotone if and only if it is operator concave
(see also Hansen-Pedersen characterizations in \cite{Hansen-Pedersen}).
One of the most beautiful and important results in operator theory is the so-called L\"{o}wner-Heinz
inequality (see \cite{Heinz}, \cite{Lowner}) which is equivalent to the operator monotonicity/concavity of the function 
$x \mapsto x^p$ on $\R^+$ when $p \in [0,1]$.

Operator monotone functions on the nonnegative reals have applications in many areas, including mathematical physics
and electrical engineering.
They arise in analysis of electrical networks (see e.g. \cite{Anderson-Trapp}) 
and in problems of elementary particles (see e.g. \cite{Wigner-von Neumann}).
They play major roles in the so-called Kubo-Ando theory of operator connections and operator means (see \cite{Kubo-Ando}).
This axiomatic theory plays an important role in operator inequalities, operator equations, 
network theory and quantum information theory.
Indeed, there is a one-to-one correspondence between operator monotone functions
on $\R^+$ and operator connections. 
Moreover, each operator mean corresponds to a unique operator monotone function on $\R^+$
which is normalized in the sense that $f(1)=1$.
See more information in \cite{Bhatia}, \cite{Donoghue}
and \cite{Hiai-Yanagi}.

Every operator monotone function from $\R^+$ to $\R^+$ admits the following integral representation:
\begin{proposition}[\cite{Lowner}] \label{thm: Lowner}
    A continuous function $f: \R^+ \to \R^+$ is operator monotone if and only if there is a unique
    finite Borel measure $\nu$ on $[0,\infty]$ such that
\begin{align}
		f(x) = \int_{[0,\infty]} \frac{x(\ld+1)}{x+\ld} \,d \nu(\ld), \quad x \in \R^+. \label{eq: int rep of OMF}
\end{align}
\end{proposition}

In this paper, we characterize operator monotone functions on $\R^+$ in terms of Borel measures on $[0,1]$.
The cone of such functions, denoted by $OM(\R^+)$, is equipped with the pointwise order.
The cone of finite Borel measures on $[0,1]$ is also equipped with pointwise order.
Denote the $t$-weighted harmonic mean by $!_t$. 
The main result of this paper states that there is a one-to-one correspondence between $OM(\R^+)$ 
and finite Borel measures on $[0,1]$ via a suitable form of integral representation:

\begin{theorem} \label{thm: int rep of f}
Given a finite Borel measure $\mu$ on $[0,1]$, the function 
\begin{align}
		f(x) = \int_{[0,1]} 1 \,!_t\, x \,d \mu(t), \quad x \geqs 0 \label{eq: int rep f}
\end{align}
is an operator monotone function from $\R^+$ to $\R^+$.
In fact, every operator monotone function from $\R^+$
to $\R^+$ arises in  this form.
Moreover, the map $\mu \mapsto f$ is bijective, affine and order-preserving.
\end{theorem}

Thus the functions $x \mapsto 1\,!_t\,x$ for $t \in [0,1]$ form building blocks 
for arbitrary operator monotone functions on $\R^+$.
Moreover, a function $f \in OM(\R^+)$ is normalized if and only if $\mu$ is a probability measure.
This means that every normalized operator monotone function on $\R^+$ can be viewed as an average of the special operator monotone functions $x \mapsto 1\,!_t\,x$ for
$t \in [0,1]$.
The integral representation \eqref{eq: int rep f} also has advantages in treating
decompositions of operator monotone functions.
It turns out that every function $f \in OM(\R^+)$ can be expressed as
\begin{align*}
	f = f_{ac} + f_{sd} + f_{sc}
\end{align*}
where $f_{ac}$, $f_{sd}$ and $f_{sc}$ also belong to the class $OM(\R^+)$.
The ``singularly discrete part" $f_{sd}$ is a countable sum of $x \mapsto 1\,!_t\,x$ for $t \in [0,1]$ with nonnegative coefficients.
The ``absolutely continuous part" $f_{ac}$ has an integral representation
with respect to Lebesgue measure $m$ on $[0,1]$.
The ``singularly continuous part" $f_{sc}$ has an integral representation with respect to a continuous measure
mutually singular to $m$.

In Section 2, we establish suitable integral representations for operator monotone functions on $\R^+$.
Their decompositions are treated in Section 3.

\section{Integral representations of operator monotone functions}

In this section, we characterize operator monotone functions from $\R^+$ to $\R^+$
in terms of Borel measures on $[0,1]$.
Indeed, every operator monotone function on $\R^+$ admits
a canonical representation as an intergral of special operator monotone functions $x \mapsto 1\,!_t\,x$ for $t \in [0,1]$
with respect to a finite Borel measures on $[0,1]$.
Here is the proof of this fact.

\begin{proof}[Proof of Theorem \ref{thm: int rep of f}]
The function $f$ in \eqref{eq: int rep f} is well-defined since $\mu$ is finite and, for each $x \geqs 0$, $1 \,!_t\, x$
is bounded by $\max(1,x)$ for all $t \in [0,1]$.
To show that $f$ is operator monotone, consider positive operators $A,B$ on a Hilbert space 
such that $A \leqs B$.
Since the weighted harmonic mean $A \,!_t\, B = [(1-t)A^{-1}+tB^{-1}]^{-1}$ are monotone, we have
\begin{align*}
	f(A) = \int_{[0,1]} I \,!_t\, A \,d \mu(t) \leqs \int_{[0,1]} I \,!_t\, B \,d \mu(t) = f(B).
\end{align*}
This show that the map $\mu \mapsto f$ is well-defined.
For injectivity of this map, let $\mu_1$ and $\mu_2$ be finite Borel measures on $[0,1]$ such that
$f_1 = f_2$ where 
\begin{align*}
	f_1(x) = \int_{[0,1]} 1\,!_t\, x\, d\mu_1(t), \quad 
    f_2(x) = \int_{[0,1]} 1 \,!_t\, x\, d\mu_2(t), \quad x \geqs 0.
\end{align*}
Then, for each $x \geqs 0$, 
\begin{align*}
    f_i (x) = \int_{[0,\infty]} \frac{x(\ld+1)}{x+\ld} \,d\mu_i \Psi(\ld) \quad i=1,2
\end{align*}
where $\Psi: [0,\infty] \to [0,1]$, $t \mapsto t/(t+1)$ and the measure $\mu_i \Psi$ is defined by
$E \mapsto \mu_i (\Psi(E))$ for each Borel set $E$.
Proposition \ref{thm: Lowner} implies that $\mu_1 = \mu_2$.

For surjectivity of this map, consider $f \in OM(\R^+)$.
By Proposition \ref{thm: Lowner}, there is a finite Borel measure $\nu$ on $[0,\infty]$ such that
\eqref{eq: int rep of OMF} holds.
Define a finite Borel measure $\mu$ on $[0,1]$ by $\mu = \nu \Psi^{-1}$.
A computation shows that
\begin{align*}
		f(x) = \int_{[0,1]} 1 \,!_t\, x \,d \mu(t), \quad x \geqs 0.
\end{align*}
Hence the map $\mu \mapsto f$ is surjective.
It is easy to see that this map is affine and order-preserving.
\end{proof}

The measure $\mu$ in the previous theorem is called the \emph{associated measure} of
the operator monotone function $f$.

\begin{remark}
	The map $f \mapsto \mu$ is not order-preserving in general. 
    Consider $f(x)=2x/(1+x)$ and $g(x)=(1+x)/2$.
	Then $\mu_f = \delta_{1/2}$ and $\mu_g = (\delta_0 + \delta_1)/2$. 
    We have $f \leqs g$ but $\mu_f \nleqslant \mu_g$.
\end{remark}

\begin{example} \label{ex: assoc measure of HM}
	The operator monotone function $x \mapsto 1 \,!_t\,x$ corresponds to the Dirac measure $\delta_t$ at $t \in [0,1]$.
	In particular, the operator monotone functions $x \mapsto 1$ and $x \mapsto x$ correspond to
	the measures $\delta_0$ and $\delta_1$, respectively.
	By affinity of the map $\mu \mapsto f$, the measure
    $\sum_{i=1}^n a_i \, \delta_{t_i}$,
    where $t_i \in [0,1]$ and $a_i\geqs 0$, is associated to the function $x \mapsto \sum_{i=1}^n a_i (1 \, !_{t_i} \, x)$.
\end{example}

\begin{example} \label{ex: assoc measure of log}
	Consider the operator monotone function $x \mapsto \log{(1+x)}$. 
	An elementary calculation shows that
	\begin{align*}
		\log{(1+x)} = \int_1^{\infty} \frac{x}{x+\ld}\,d\ld.
	\end{align*}
	Hence its associated measure is given by $\displaystyle \frac{1}{t} \chi_{[1/2,1]}(t) \,dt$.
\end{example}

\begin{example} \label{ex: assoc measure of weight geometric}
  Consider the associated measure of the operator monotone function $x^{\ap}$ for $0<\ap<1$.
A standard result from contour integrations says that
\begin{align*}
    x^{\ap} =  \int_{[0,\infty]} \frac{x \ld^{\ap -1}}{x+\ld} \cdot \frac{\sin \ap \pi}{\pi}  \, d\ld.
\end{align*}
It follows that the associated measure of this function is given by
\begin{align*}
    d\mu(t) = \frac{\sin \ap \pi}{\pi} \cdot  \frac{1}{t^{1-\ap} (1-t)^{\ap}} \,dt.
\end{align*}
\end{example}

\begin{example}\label{ex: assoc measure of LM}
	Let us compute the associated measure of  the operator monotone function
	$f(x) = (x-1) / \log{x}$.
	By Example \ref{ex: assoc measure of weight geometric} and Fubini's theorem, we have
	\begin{align*}
		f(x) &= \int_0^1 x^{\ld} \,d \ld
			\, = \, \int_0^1 \int_0^1 \frac{\sin{\ld \pi}}{\pi} \cdot
                    \frac{1 \,!_{t}\, x}{t^{1- \ld}(1-t)^{\ld}} \,dt \,d \ld \\
			&= \int_0^1 (1 \,!_{t}\, x) \int_0^1 \frac{\sin{\ld \pi}}{\pi t^{1- \ld}(1-t)^{\ld}}
			\,d\ld \,dt.
	\end{align*}
	Hence the associated measure of $f$ is the measure $d \mu(t) = g(t) \,dt$ where
	the density function $g$ is given by
	\begin{align*}
		g(t) &= \int_0^1 \frac{\sin{\ld \pi}}{\pi t^{1-\ld} (1-t)^{\ld}} \,d\ld
				 = \frac{1}{t(1-t) \left(\pi^2 + \log^2 (\frac{t}{1-t}) \right) }.
	\end{align*}
	Now, consider the \emph{dual} of the function $f$, defined by
	\begin{align*}
		x \mapsto  \frac{1}{f(1/x)} = \frac{x}{x-1} \log{x}.
	\end{align*}
	This operator monotone function has Lebesgue measure as the associated measure, equivalently,
	we have the integral representation
	\begin{align*}
			\frac{x}{x-1} \log{x} = \int_{[0,1]} 1 \,!_t x \,dt.
	\end{align*}
\end{example}

\begin{remark}
Theorem \ref{thm: int rep of f} implies a short proof of the fact that a continuous function $f: \R^+ \to \R^+$ 
is operator monotone function if and only if it is operator concave as follows.
If $f$ is operator monotone, then it has the integral representation \eqref{eq: int rep f}
which implies that $f$ is operator concave since the weighted harmonic means are concave.
Now, suppose that $f$ is operator concave. Consider positive operators $A,B$ on a Hilbert space and let $t \in [0,1]$. 
				Since $tA = tB+(1-t)t(1-t)^{-1}(A-B)$, the operator concavity implies that
				\begin{align*}
					f(tA) \geqs tf(B)+(1-t)f(t(1-t)^{-1}(A-B)).
				\end{align*}
				Since $f \geqs 0$, we have $f(t(1-t)^{-1}(A-B) \geqs 0$ and $f(tA) \geqs f(tB)$. 
				Letting $t \nearrow 1$ yields $f(A) \geqs f(B)$.
\end{remark}

	A function $f\in OM(\R^+)$ is said to be normalized if $f(1)=1$.

\begin{corollary}
	The functions $x \mapsto 1 \,!_t\, x$ for $t \in [0,1]$ are extreme points of the convex set of normalized
	operator monotone functions from $\R^+$ to $\R^+$.
\end{corollary}
\begin{proof}
    It is easy to see that the set of normalized operator monotone functions is closed under convex combinations. 
	To obtain extreme points, apply Theorem \ref{thm: int rep of f} to the fact that 
    the Dirac measures are extreme points of the convex set of probability Borel measures on $[0,1]$.
\end{proof}

\begin{corollary} \label{cor: mean iff prob meas}
	An operator monotone function $f: \R^+ \to \R^+$ is normalized if and only if
	its associated measure is a probability measure.
	Hence	there is a one-to-one correspondence between normalized operator monotone function from $\R^+$ to $\R^+$ 
	and probability Borel measures on $[0,1]$ via the integral representation \eqref{eq: int rep f}.
\end{corollary}
\begin{proof}
    It follows from the integral representation \eqref{eq: int rep f} in Theorem \ref{thm: int rep of f}. 
\end{proof}

This corollary means that every normalized operator monotone function on $\R^+$ can be regarded as an average of the 
special operator monotone functions $x \mapsto 1\,!_t\,x$ for
$t \in [0,1]$.

\begin{corollary}
	The followings are equivalent for a function $f:\R^+ \to \R^+$:
	\begin{enumerate}
		\item[(i)]	$f$ is operator monotone;
		\item[(ii)]	$g$ is operator monotone where $g(x)=xf(1/x)$ for $x>0$ and $g(0)=\lim_{x \downarrow 0} g(x)$;
	\end{enumerate}
Moreover, the associated measure of $g$ is given by
$\mu \Theta $ where $\Theta  : [0,1] \to [0, 1]$, $t \mapsto 1-t$.
\end{corollary}

\begin{proof}
		(i) $\Rightarrow$ (ii). 	From the integral representation \eqref{eq: int rep f} of $f$,
	we have
	\begin{align*}
		x f(\frac{1}{x})
		&= x \int_{[0,1]} 1 \,!_t\, \frac{1}{x}\, d\mu(t)
		= \int_{[0,1]} x \,!_t\, 1\, d\mu(t) \\
		&= \int_{[0,1]} x \,!_{1-t} \, 1 \, d\mu(1-t) = \int_{[0,1]} 1 \,!_{t}\, x\, d\mu\Theta(t)
	\end{align*}
	for each $x>0$. For $x=0$, use continuity. By Theorem \ref{thm: int rep of f}, 
	$g$ is operator monotone with associated measure $\mu \Theta$.
	
	(ii) $\Rightarrow$ (i). Since $f(x)=xg(1/x)$, we can apply the implication (i) $\Rightarrow$ (ii).
\end{proof}

	The function $g$ in this corollary is called the \emph{transpose} of $f$.
	We say that $f$ is \emph{symmetric} if it coincides with its transpose.
	A Borel measure $\mu$ on $[0,1]$ is said to be \emph{symmetric} if $\mu$ is invariant under $\Theta$,
	i.e. $\mu \Theta = \mu$.

\begin{corollary}	\label{cor: symmetric OMF}
	There is a one-to-one correspondence between symmetric operator monotone function from $\R^+$ to $\R^+$
	 and finite symmetric Borel measures
	 via the integral representation
	\begin{align}
		f(x) = \frac{1}{2} \int_{[0,1]} (1 \,!_t\, x)  +  (x \,!_t\, 1) \,d \mu(t), \quad x \geqs 0.
		 \label{int rep of symmetric OM}
	\end{align}
	In particular, an operator monotone function on $\R^+$ is symmetric if and only if 
    its associated measure is symmetric.
\end{corollary}

\begin{corollary} \label{cor: normalized symmetric OMF}
	There is a one-to-one correspondence between normalized symmetric operator monotone functions on 
    $\R^+$ and probability symmetric Borel measures on the unit interval via the integral representation 
    \eqref{int rep of symmetric OM}.
\end{corollary}

\begin{corollary}
	The functions $x \mapsto (1 \, !_t\, x \,+\, x \, !_t\, 1)/2$ for $t \in [0,1]$ 
    are extreme points of the convex set of
	normalized symmetric operator monotone function from $\R^+$ to $\R^+$.
\end{corollary}
\begin{proof}
    It is easy to see that the measures $(\delta_t + \delta_{1-t})/2$ for $t \in [0,1]$ 
    are extreme points of the convex set of
	probability symmetric Borel measures on $[0,1]$.
	Use this fact and the fact that the map sending operator monotone functions 
    to associated measures is bijective and affine.
\end{proof}

\section{Decomposition of operator monotone functions}

In this section, we decompose operator monotone functions from $\R^+$ to $\R^+$.

\begin{theorem} \label{thm: decomp of func}
    For each $f \in OM(\R^+)$, 
    there is a unique triple $(f_{ac},f_{sc},f_{sd})$ of operator monotone functions on $\R^+$ such that
\begin{align*}
    f = f_{ac} + f_{sc} + f_{sd}
\end{align*}
and
\begin{enumerate}
    \item[(i)] there are a countable set $D \subseteq [0,1]$ and a family
    						$\{a_{t}\}_{t \in D} \subseteq \R^+$
                 such that $\sum_{t \in D} a_{t} < \infty$ and for each $x \in \R^+$
            \begin{align}
                f_{sd}(x) = \sum_{\ld \in D} a_{t} (1 \,!_t\, x); \label{eq: formula of f_sd}
            \end{align}
    \item[(ii)]   there is a (unique $m$-a.e.) integrable function $g: [0,1] \to \R^+$ such that
            \begin{align}
                f_{ac} (x) = \int_{[0,1]} g(t)(1 \,!_t\, x) \,dm(t), \label{eq: formula of f_ac}
                \quad x \in \R^+;
            \end{align}
    \item[(iii)]  its associated measure of $f_{sc}$ is continuous and mutually singular to $m$.
\end{enumerate}
Moreover, the associated measure of $f_{sd}$ is given by $\sum_{t \in D} a_{t} \,\delta_{t}$.
\end{theorem}
\begin{proof}
    Let $\mu$ be the associated measure of $f$. Then there is a unique triple $(\mu_{ac},\mu_{sc},\mu_{sd})$ 
    of finite Borel measures on $[0,1]$ such that $\mu= \mu_{ac} + \mu_{sc} + \mu_{sd}$ where
(1) $\mu_{sd}$ is a discrete measure, (2) $\mu_{ac}$ is absolutely continuous with respect to $m$  and
(3) $\mu_{sc}$ is a continuous measure mutually singular to $m$.
We can write
\begin{align*}
	f(x) = \int_{[0,1]} 1\,!_t\, x\, d\mu_{ac} + \int_{[0,1]} 1\,!_t\, x\, d\mu_{sd}
	+ \int_{[0,1]} 1\,!_t\, x\, d\mu_{sc}, \quad x \in \R^+.
\end{align*}
Since $\mu_{ac}, \mu_{sd}$ and $\mu_{sc}$ are finite Borel measures on $[0,1]$, they give rise
to $f_{ac}, f_{sd},f_{sc}$ in $OM(\R^+)$, given by
\begin{align*}
	f_{ac} (x) &= \int_{[0,1]} 1\,!_t\, x\, d\mu_{ac},  \\
	f_{sd} (x) &= \int_{[0,1]} 1\,!_t\, x\, d\mu_{sd}, \\
	f_{sc} (x) &= \int_{[0,1]} 1\,!_t\, x\, d\mu_{sc}.
\end{align*}
The condition (1) means precisely that there are a countable set $D \subseteq [0,1]$ and a family
$\{a_{\ld}\}_{\ld \in D}$ in $\R^+$
such that $\sum_{\ld \in D} a_{\ld} < \infty$ and
$
    \mu_{sd} = \sum_{t \in D} a_{t} \delta_{t}.
$
Hence,
$
	f_{sd} (x) = \sum_{t \in D} a_{t} (1 \,!_t\, x)
$
for $x \in \R^+$.
Note that this series converges since, for each $x \in \R^+$,
\begin{align*}
	\sum_{t \in D} a_{t} (1 \,!_t\, x) \leqs \sum_{t \in D} a_{t} \max( 1,x ) < \infty.
\end{align*}
The condition (2) means precisely the condition (ii) by Radon-Nikodym theorem.
The uniqueness of $(f_{ac},f_{sc},f_{sd})$ follows from the uniqueness of $(\mu_{ac},\mu_{sc},\mu_{sd})$
and the correspondence between operator monotone functions and measures.
The measure $\sum_{t \in D} a_{t} \,\delta_{t}$ is associated to $f_{sd}$ since
the associated measure of $x \mapsto 1\,!_t\,x$ is $\delta_t$ for each $t \in [0,1]$ by Example
\ref{ex: assoc measure of HM}.
\end{proof}

This theorem says that every $f \in OM(\R^+)$ consists of three parts.
The ``singularly discrete part" $f_{sd}$ appears in terms of weighted harmonic means.
Indeed, it is a countable sum of $x \mapsto 1\,!_t x$ for $t \in [0,1]$, given by \eqref{eq: formula of f_sd}.
Such type of functions include the straight lines with positive slopes, the constant functions
and the multiple functions $x \mapsto kx$.
The ``absolutely continuous part" $f_{ac}$ arises as an integral
with respect to Lebesgue measure, given by the formula \eqref{eq: formula of f_ac}.
Typical examples of such functions are the following functions:
\begin{align*}
    \log{(1+x)}, \quad x^{\ap}, \quad \frac{x-1}{\log{x}}, \quad \frac{x}{x-1} \log{x},
\end{align*}
provided in Examples \ref{ex: assoc measure of log}, \ref{ex: assoc measure of weight geometric}  
and \ref{ex: assoc measure of LM}.
The ``singularly continuous part" $f_{sc}$ has an integral representation with respect to a continuous measure
mutually singular to Lebesgue measure.
Hence (aside singular continuous part) this theorem gives an explicit description of arbitrary operator monotone
functions on the nonnegative reals.

\begin{proposition} \label{prop: properties of f ac}
Consider the operator monotone function $f_{ac}$ defined by \eqref{eq: formula of f_ac}. Then
    \begin{itemize}
        \item   it is normalized if and only if the average of the density function $g$ is $1$, i.e.
                $
	               \int_0^1 g(t)\,dt = 1.
                $
        \item   it is symmetric if and only if $g \circ \Theta = g$.
    \end{itemize}
\end{proposition}
\begin{proof}
    Use Corollaries \ref{cor: mean iff prob meas} and \ref{cor: symmetric OMF}.
\end{proof}

We say that a density function $g:[0,1] \to \R^+$ is \emph{symmetric} if $g \circ \Theta = g$.
Next, we decompose a normalized operator monotone function as a convex combination 
of normalized operator monotone functions.

\begin{corollary} \label{cor: normalized OMF decom}
    Let $f \in OM(\R^+)$ be normalized.
Then there are unique triples $(\widetilde{f_{ac}},\widetilde{f_{sc}},\widetilde{f_{sd}})$
of normalized operator monotone functions or zero functions and $(k_{ac}, k_{sc}, k_{sd})$ 
of real numbers in $[0,1]$ such that
\begin{align*}
    f = k_{ac} \widetilde{f_{ac}} + k_{sc} \widetilde{f_{sc}} + k_{sd} \widetilde{f_{sd}},
    \quad k_{ac} + k_{sc} + k_{sd} = 1
\end{align*}
and
\begin{enumerate}
    \item[(i)] there are a countable set $D \subseteq [0,1]$ and a family
    						$\{a_{t}\}_{t \in D} \subseteq [0,1]$
                 such that $\sum_{t \in D} a_{t} =1$ and 
            $
                f_{sd}(x) = \sum_{\ld \in D} a_{t} (1 \,!_t\, x) 
            $ for each $x \in \R^+$;
    \item[(ii)]   there is a (unique $m$-a.e.) integrable function $g: [0,1] \to \R^+$ with average $1$ such that
            $
                f_{ac} (x) = \int_{[0,1]} g(t)(1 \,!_t\, x) \,dm(t)   
            $ for $x \in \R^+$;
    \item[(iii)]  its associated measure of $f_{sc}$ is continuous and mutually singular to $m$.
\end{enumerate}
\end{corollary}
\begin{proof}
	Let $\mu$ be the associated probability measure of $f = f_{ac} + f_{sd} + f_{sc}$
	and write $\mu = \mu_{ac} + \mu_{sd} + \mu_{sc}$.
	Suppose that $\mu_{ac}$, $\mu_{sd}$ and $\mu_{sc}$ are nonzero.
	Then
	\begin{align*}
		\mu = \mu_{ac} ([0,1]) \frac{\mu_{ac}}{\mu_{ac} ([0,1])} + \mu_{sd} ([0,1]) \frac{\mu_{sd}}{\mu_{sd} ([0,1])}
				+ \mu_{sc} ([0,1]) \frac{\mu_{sc}}{\mu_{sc} ([0,1])}.
	\end{align*}
	Set
	\begin{align*}
		\widetilde{\mu_{a}} &= \frac{\mu_{ac}}{\mu_{ac} ([0,1])}, \quad \widetilde{\mu_{sd}}
        = \frac{\mu_{sd}}{\mu_{sd} ([0,1])},
			\quad \widetilde{\mu_{sc}} = \frac{\mu_{sc}}{\mu_{sc} ([0,1])}, \\
		k_{ac} &= \mu_{ac} ([0,1]), \quad k_{sd} = \mu_{sd} ([0,1]), \quad k_{sc} = \mu_{sc} ([0,1]).
	\end{align*}
	Define $\widetilde{f_{ac}}, \widetilde{f_{sd}}, \widetilde{f_{sc}}$ to be the functions corresponding to
	the measures $\widetilde{\mu_{ac}}, \widetilde{\mu_{sd}}, \widetilde{\mu_{sc}}$, respectively.
	Now, apply Theorem \ref{thm: decomp of func} and Proposition
    \ref{prop: properties of f ac}.
\end{proof}

We can decompose a symmetric operator monotone function
as a nonnegative linear combination of symmetric operator monotone functions as follows:

\begin{corollary} \label{cor: symmetric OMF decom}
    Let $f \in OM(\R^+)$ be symmetric.
Then there is a unique triple $(f_{ac},f_{sc},f_{sd})$ of symmetric operator monotone functions such that
\begin{align*}
    f = f_{ac} + f_{sc} + f_{sd} \label{eq: sym con decomposition}
\end{align*}
and
\begin{enumerate}
    \item[(i)] 
    there are a countable set $D \subseteq [0,1]$ and a family $\{a_{t}\}_{t \in D} \subseteq \R^+$ such that
    $a_t = a_{1-t}$ for all $t \in D$, $\sum_{t\in D} a_t < \infty$ and
            $
                f_{sd}(x) = \sum_{t \in D} a_{t} (1 \,!_{t}\, x)
            $ for each $x \in \R^+$;
    \item[(ii)]   there is a (unique $m$-a.e.) symmetric integrable function $g: [0,1] \to \R^+$ such that
            \begin{align*}
                f_{ac} (x) = \frac{1}{2} \int_{[0,1]} g(t)
                ( 1 \,!_{t}\, x \,+\, x \,!_{t}\, 1) \,dm(t), \quad x \geqs 0;
            \end{align*}
    \item[(iii)] its associated measure of $f_{sc}$ is continuous and mutually singular to $m$.
\end{enumerate}
\end{corollary}
\begin{proof}
	Let $\mu$ be the associated measure of $f$ and decompose
	$\mu = \mu_{ac} + \mu_{sd} + \mu_{sc}$ where
	$\mu_{ac} \ll m$, $\mu_{sd}$ is a discrete measure, $\mu_{sc}$ is continuous
	and $\mu_{sc} \perp m$.
	Then $\mu \Theta = \mu_{ac} \Theta + \mu_{sd} \Theta + \mu_{sc} \Theta$.
	It is straightforward to show that
	$\mu_{ac} \Theta \ll m$, $\mu_{sd} \Theta$ is discrete, $\mu_{sc} \Theta$ is continuous and $\mu_{sc} \Theta \perp m$.
	By Corollary \ref{cor: symmetric OMF}, $\mu \Theta = \mu$.
	The uniqueness of the decomposition of measures implies that $\mu_{ac} \Theta = \mu_{ac}, \mu_{sd} \Theta = \mu_{sd}$
	and $\mu_{sc} \Theta = \mu_{sc}$.
	Again, Corollary \ref{cor: symmetric OMF} tells us that $f_{ac}, f_{sd}$
	and $f_{sc}$ are symmetric operator monotone functions.
	Finally, apply Theorem \ref{thm: decomp of func} and Proposition
    \ref{prop: properties of f ac}.
\end{proof}

A decomposition of any normalized symmetric operator monotone function
as a convex combination of such ones is also obtained
by the normalizing process in the proof of 
Corollary \ref{cor: normalized OMF decom}.

\bibliographystyle{amsplain}

\end{document}